\documentclass[plain]{amsart}
\usepackage[toc,page]{appendix}
\usepackage{amssymb,graphicx,epsfig,amsxtra, amsmath}
\usepackage{amscd} 
\usepackage{xypic}
\usepackage[mathscr]{eucal}
\usepackage{a4}
\input{xy}
\xyoption{all}
\newtheorem{theorem}{\textbf{Theorem}}[section]
\newtheorem{proposition}[theorem]{\textbf{Proposition}}
\newtheorem{definition}[theorem]{\textbf{Definition}}
\newtheorem{lemma}[theorem]{\textbf{Lemma}}
\newtheorem{corollary}[theorem]{\textbf{Corollary}}
\newtheorem{remark}[theorem]{\textbf{Remark}}

\newtheorem{definition-remark}{\textbf{Definition-Remark}}

\def\ag{\`a}

\def\Pic{\operatorname{Pic}}
\def\Sing{\operatorname{Sing}}

\def\p{\rm{p}}
\def\N{\mathcal N}
\def\d{\delta}

\begin{document}

\title[Curves with a triple point on a $K3$ surface]{On the existence of curves with a triple point on a $K3$ surface}
\author{Concettina Galati }
\address{Dipartimento di Matematica\\
 Universit\ag\, della Calabria\\
via P. Bucci, cubo 31B\\
87036 Arcavacata di Rende (CS), Italy. }
\email{galati@mat.unical.it }

\thanks{The author was partially supported by GNSAGA of INdAM
and by the PRIN 2008 'Geometria delle variet\ag\,  algebriche e dei loro spazi di moduli', co-financied by MIUR}



\date{23.02.12}

\dedicatory{}

\commby{}


\begin{abstract}
Let $(S,H)$ be a general primitively polarized $K3$ surface 
of genus $\p$ and let  $p_a(nH)$ be the arithmetic genus of $nH.$ 
We prove the existence in $|\mathcal O_S(nH)|$ of curves
with a triple point and $A_k$-singularities. In particular, we show the existence
 of curves of geometric genus $g$  in $|\mathcal O_S(nH)|$ 
with a triple point and nodes as singularities and corresponding to regular
points of their equisingular deformation locus, for every $1\leq g\leq p_a(nH)-3$
and $(\p,n)\neq (4,1).$
Our result is obtained by studying the 
versal deformation space of a non-planar quadruple point.
\end{abstract}


\maketitle

\section{Introduction}
Let $S$ be a complex smooth projective $K3$ surface and let $H$ be a very ample 
line bundle on $S$ of sectional genus ${\p} =p_a(H) \geq 2.$ 
 The pair $(S,H)$ is called a {\it  primitively polarized 
$K3$ surface of genus $\p.$}  It is well-known that a (very) general such pair satisfies 
$\Pic S \cong \mathbb Z[H].$
We denote
by $\mathcal V_{nH,\delta}^S\subset |\mathcal O_S(nH)|=|nH|$ the {\it Zariski closure} of the Severi variety
of $\delta$-nodal curves, defined as the locus of
{\it irreducible} and reduced curves with exactly 
$\delta$ nodes as singularities. The non-emptiness of these varieties for every $\delta\leq\dim(|nH|)=p_a(nH),$ 
where $p_a(nH)$ is the arithmetic genus of $nH,$ has been established in \cite{C}.
Like the Severi variety of $\delta$-nodal plane curves, $\mathcal V_{nH,\delta}^S$
has several good properties.   By \cite{tan}, we know that  it
is smooth of expected dimension at every point $[C]$ corresponding to 
a $\d$-nodal curve. This implies that every irreducible component $V$ of $\mathcal V_{nH,\delta}^S$ has codimension
$\d$ in $|nH|$ and it is contained in at least one irreducible component of the Severi variety $\mathcal V_{nH,\delta-1}^S.$  
Moreover, it is also known that  $\mathcal V_{nH,\delta}^S$ coincides with the variety 
$\mathcal V_{nH,g}^S\subset |nH|,$ defined as the Zariski closure of the locus of reduced and irreducible curves of
geometric genus $g$ (cf. \cite[Lemma 3.1]{C} and \cite[Remark 2.6]{galati-knutsen}).  
Unlike the Severi variety of $\delta$-nodal plane curves, nothing is known about the irreducibility
of $\mathcal V_{nH,\delta}^S.$ Classically, the irreducibility problem of Severi varieties is related to the
problem of the description of their boundary (cf. \cite{harris}).  
\begin{trivlist}
\item[\hspace\labelsep{\bf  Problem 1}]
{\em Let $V\subset \mathcal V_{nH,\delta}^S$ be an irreducible component and  let 
$V^o\subset V$ be the locus of $\delta$-nodal curves. 
What is inside the boundary $V\setminus V^o$?  
 Does $V$  contain divisors 
$V_{n}$, $V_{tac}$, $V_{c}$ and $V_{tr}$ whose general element corresponds to a  
curve with $\delta+1$ nodes, a tacnode and $\delta-2$ nodes, a cusp and $\delta-1$
nodes and a triple point and $\delta-3$ nodes, respectively?}
\end{trivlist}
Because of the literature about Severi varieties of plane curves (cf. \cite{diaz_harris1}, \cite{diaz_harris2} and \cite{F}), the divisors 
$V_{n}$, $V_{tac}$, $V_{c}$ and $V_{tr}$ of $V\subset \mathcal V_{nH,\delta}^S$, when non-empty,
are expected to play an important role in the description of the Picard group $\Pic(V)$ or, more precisely,
of the Picard group of a "good partial compactification" of the locus $V^o\subset V$ of $\delta$-nodal curves.  
Proving the existence of these four divisors in {\em every} irreducible component $V$ of  $\mathcal V_{nH,\delta}^S$
is a very difficult problem. The first progress towards answering the previous question has been made in \cite{galati-knutsen}.
In that paper, the authors prove the existence of irreducible curves in $|nH|$ of every allowed genus with a tacnode or a cusp
and nodes as further singularities.  This article is devoted to the existence of irreducible curves in $|nH|$ of geometric genus 
$1\leq g\leq p_a(nH)-3$ with a triple point and nodes as further singularities.  Before introducing our result,
we make some observations concerning Problem 1, describing the type of singularity of $V$ along
$V_{n}$, $V_{tac}$, $V_{c}$ and $V_{tr},$ whenever  these loci are non-empty.
Let us denote by  $\mathcal V_{nH,g,tac}^S$, $\mathcal V_{nH,g,c}^S$ and
$\mathcal V_{nH,g,tr}^S$ the Zariski closure of the locus in $|nH|$ of reduced and irreducible curves of geometric genus $g$
with a tacnode,  a cusp and a triple point, respectively, and nodes as further singularities. Let $W$ be an irreducible component 
of any of these varieties or of $\mathcal V_{nH,g-1}^S.$ Then, by
\cite[Lemma 3.1]{C}, we have that
$\dim(W) =g-1.$ Thus $W$ is a divisor in at least one irreducible component $V$ of 
$\mathcal V_{nH,\delta=p_a(nH)-g}^S=\mathcal V_{nH,g}^S$ and, by using the same arguments as 
in \cite[Section 1]{diaz_harris}, we also know what $V$
looks like in a neighborhood of the general point of $W.$  If $W$ is an irreducible component 
of $\mathcal V_{nH,g-1}^S,$ then $\mathcal V_{nH,g}^S$ has an ordinary multiple point of order $\delta+1$
at the general point of $W.$ In particular, at least in principle, there may not be a unique irreducible component $V$ of $\mathcal V_{nH,g}^S$ containing $W.$ On the contrary, if $W$ is any irreducible component of 
$\mathcal V_{nH,g,tac}^S$, $\mathcal V_{nH,g,c}^S$ or
$\mathcal V_{nH,g,tr}^S,$
then $W$ is contained in only one irreducible
component $V$ of $\mathcal V_{nH,g}^S$. 
In particular, $\mathcal V_{nH,g}^S$ is smooth at the general point of every irreducible 
component of $\mathcal V_{nH,g,tac}^S$ and $\mathcal V_{nH,g,tr}^S$ and it looks like 
the product of a cuspidal curve and a smooth $(g-1)$-dimensional variety in a neighborhood of the general point
of an irreducible component of $ \mathcal V_{nH,g,c}^S.$
We finally observe that, unlikely
the case of Severi varieties of plane curves, if $V\subset\mathcal V_{nH,g}^S$ is an irreducible component, 
a priori, there may exist divisors in $V$
parametrizing curves with worse singularities than an ordinary triple point, a tacnode or a cusp and nodes. 
We now describe the results of this paper. Our theorem about
non-emptiness of $\mathcal V_{nH,g,tr}^S$ is based on the following local problem.
\begin{trivlist}
\item[\hspace\labelsep{\bf  Problem 2}]
{\em  Let $\mathcal X\to\mathbb A^1$ be a smooth family of surfaces with smooth general fiber
$\mathcal X_t$ and whose special fiber $\mathcal X_0=A\cup B$
is reducible, having two irreducible components intersecting transversally along a 
smooth curve $E=A\cap B.$ What kind of curve singularities on $\mathcal X_0$ at a point $p\in E$
may be deformed to a triple point on $\mathcal X_t$? }
\end{trivlist}
In the first part of Section \ref{sect: getting triple point}, we prove that, 
if a triple point on $\mathcal X_t$ specializes to a general point $p\in E$
along a smooth bisection $\gamma$ of $\mathcal X\to\mathbb A^1,$ then the limit curve singularity on 
$\mathcal X_0$ is a space quadruple point, union of two nodes having one branch tangent to $E.$
The result is obtained by a very simple argument of limit linear systems theory, with the same techniques 
as in \cite{galati}. In Lemma \ref{lemma: analytic-type}
we find the analytic type of this quadruple point. This allows us to compute the versal deformation space
of our limit singularity and to prove that, under suitable hypotheses, it actually deforms to an ordinary triple point singularity
on $\mathcal X_t,$ see Theorem \ref{th: triple-point}. In the last section, we consider
the case that $\mathcal X_t$ is a general primitively
polarized $K3$ surface and $A=R_1$ and $B=R_2$ are two rational normal scrolls. In Lemma \ref{limiti},
we prove the existence on $\mathcal X_0=R_1\cup R_2$  of suitable curves with a non-planar quadruple
point as above, tacnodes (of suitable multiplicities) and nodes and that may be deformed to curves in $|\mathcal O_{\mathcal X_t}(nH)|$
with an ordinary triple point and nodes as singularities. In particular, this existence result is obtained as a corollary
of the following theorem, which is to be considered the main theorem of this paper. 
\begin{theorem} \label{main-theorem}
Let $(S,H)$ be a general primitively polarized $K3$ surface of genus ${\p}=p_a(H).$ 
 Then, for every $(\p,n)\neq (4,1)$ and for every $(m-1)$-tuple of non-negative integers $d_2,\ldots,d_m$ such that
 \begin{equation}\label{ppari}
\sum_{k=2}^m(k-1)d_k=n(\p-2)-3=2nl-2n-3
 \end{equation}
if $\p=2l$ is even, or 
\begin{equation}\label{pdispari}
\sum_{k=2}^m(k-1)d_k=n(\p-1)-4=2nl-4,
 \end{equation}
 if $\p=2l+1$ is odd, there exist reduced irreducible curves $C$ in the linear system $|nH|$ on $S$ 
 such that:
 \begin{itemize}
 \item $C$ has an ordinary triple point, $p_a(nH)-\sum_{k=2}^m(k-1)d_k-4$ nodes and $d_k$ singularities
 of type $A_{k-1},$ for every $k=2,\ldots,m,$ and no further singularities;
 \item  $C$ corresponds to a regular point of the equisingular deformation locus $ES(C).$ Equivalently, 
 $\dim(T_{[C]}ES(C))=0.$ 
 \end{itemize} 
 Finally, the singularities of $C$ may be smoothed independently. In particular, under the hypotheses \eqref{ppari}
 and \eqref{pdispari}, for every $\delta_k\leq d_k$ and for every $\delta\leq \dim(|nH|)-\sum_{k=2}^m(k-1)d_k-4,$ 
 there exist curves $C$ in the linear system $|nH|$ on $S$ with an ordinary triple point, $\delta_k$ singularities of type $A_{k-1},$
 for every $k=2,...,m,$ and $\delta$ nodes as further singularities and corresponding to regular points of their equisingular 
 deformation locus.
\end{theorem}
By Corollary \ref{primitive-corollary}, the family in 
$|H|$ of curves with a triple point and $\delta_k$ singularities of type $A_{k-1}$
is non-empty if it has expected dimension at least equal to one or it has expected dimension equal to 
zero and $\delta_2\geq 1.$
\subsection*{Acknowledgments} 
I would like to thank M. Giusti for offering to send me personal
notes of \cite{giusti} during the holiday time of my library. My intention to write this paper became
stronger after a conversation with C. Ciliberto about related topics. I also benefited from 
conversation with T. Dedieu and A. L. Knutsen. Finally, I would like to thank the referee for
many helpful comments and corrections.

\section{Notation and terminology}\label{notation-terminology}
Throughout the paper an irreducible curve $C$ will be a reduced and irreducible projective curve,
even if not specified. We will be concerned here only with reduced and locally complete intersection curves
 $C$ in a linear system $|D|$ on a (possibly singular) surface $X,$ having, as singularities, space
 singularities, plane ordinary triple points or plane singularities of type $A_k.$
 We recall that an ordinary triple point has analytic equation $y^3=x^3$ while an $A_k$-singularity has 
 analytic equation $y^2-x^{k+1}.$ Singularities of type $A_1$ are nodes,
$A_2$-singularities are cusps and an $A_{2m-1}$-singularity is a so called $m$-tacnode.
Whenever not specified, a tacnode will be a $2$-tacnode, also named a simple tacnode. 
For curves with ordinary plane triple points and plane singularities of type $A_k,$  equisingular
deformations from the analytic point of view coincide with equisingular deformations from the Zariski 
point of view (cf. \cite[Definition (3.13)]{diaz_harris}). Because of this,  given a curve as above,
we define the equisingular deformation locus $ES(C)\subset |D|$ as the Zariski closure of the locus of 
analytically equisingular deformations
of $C$, without discrepancy with classical terminology. 
Finally, we recall that, if $X$ is a smooth surface and $C$ has as singularities an ordinary triple
point and $d_k$ singularities of type $A_k,$ then the dimension of the tangent space to $ES(C)$ at the
point $[C]$ corresponding to the curve $C$ is given by $\dim(T_{[C]}ES(C))\geq\dim|D|-4-\sum_kd_k k$
(cf. \cite{diaz_harris}).
If equality holds, we say that {\em $ES(C)$ is regular at } $[C].$ The regularity of 
 $ES(C)$ is a sufficient and necessary condition for the surjectivity 
of  the standard morphism
$H^0(C,\N_{C|X})\to T^1_C.$ In this case, we say that {\em the singularities
of $C$ may be smoothed independently}, with obvious meaning because of the versality properties of 
$T^1_C.$

\section{How to obtain curves with a triple point on a smooth surface}\label{sect: getting triple point}
Let $\mathcal X\to\mathbb A^1$ be a smooth family of projective complex surfaces with smooth general fiber
$\mathcal X_t$ and whose special fiber $\mathcal X_0=A\cup B$
has two irreducible components, intersecting transversally along a 
smooth curve $E=A\cap B.$ Let $D\subset\mathcal X$ be a Cartier divisor  and let us set
$D_t=D\cap\mathcal X_t.$ 
 We want to find sufficient conditions for the existence 
of curves with a triple point in $|D_t|=|\mathcal O_{\mathcal X_t}(D_t)|.$
We first ask the following question.
Assume that, for general $t$, there exists a reduced and irreducible divisor $C_t\in 
|D_t|$ with a triple point.
Assume, moreover, that the curve $C_t$ degenerates to 
a curve $C_0\in |D_0|=|\mathcal O_{\mathcal X_0}(D_0)|$ in such a way that the triple point of $C_t$
comes to a general point $p\in E=A\cap B.$ We ask for the type of singularity of $C_0$ at $p.$
Actually we do not want to find all possible curve singularities at $p$ that are limit of a triple
point on $\mathcal X_t.$ We only want to find a suitable limit singularity. We first observe that,
since $p$ belongs to the singular locus of $\mathcal X_0$ and $\mathcal X$ is smooth at $p,$
there are no sections of $\mathcal X\to\mathbb A^1$ passing through $p.$ Thus the triple point of 
the curve $C_t,$ as $t\to 0,$ must move along a multisection $\gamma^\prime$ of  $\mathcal X\to\mathbb A^1.$ 
In order to deal with a divisor $S$ in a smooth family of surfaces $\mathcal Y\to\mathbb A^1$
with a triple point at the general point of a section of $\mathcal Y\to\mathbb A^1,$ we make a base change of 
$\mathcal X\to\mathbb A^1.$\smallskip

Let $\mathcal Y\to\mathbb A^1$ be the smooth family of surfaces
\begin{displaymath}
\xymatrix{
\mathcal Y\ar[r]\ar[dr]&  \mathcal X^\prime \ar[d]\ar[r] &
\mathcal X\ar[d]\\
& \mathbb A^1\ar[r]^{\nu_2}&\mathbb A^1  }
\end{displaymath}
obtained from $\mathcal X\to\mathbb A^1$ after a base change of order two totally ramified at $0$
and by normalizing the achieved family.  
Now the family $\mathcal Y$ has general fibre $\mathcal Y_t\simeq\mathcal X_t$ and special 
fibre $\mathcal Y_0=A\cup \mathcal E\cup B$, where, by abusing notation, $A$ and $B$ are 
the proper transforms of $A$ and $B$ in $\mathcal X$ and $\mathcal E$ is a $\mathbb P^1$-bundle 
on $E.$ In particular, $A\cap\mathcal E$ and $B\cap\mathcal E$ are two sections of $\mathcal E$ 
isomorphic to $E.$ Denote by $F$ the fibre $\mathcal E$ corresponding to the point $p\in E\subset \mathcal X_0$ and
let $\gamma$ be a section of $\mathcal Y$ intersecting $F$ at a general point $q.$ 
Assume there exists a divisor $S\subset \mathcal Y$ having a triple point at the general 
point of $\gamma.$ 

{\em Step 1.} Let $\pi_1:\mathcal Y^1\to\mathcal Y$ be the blow-up of $\mathcal Y$ along $\gamma$ 
with new exceptional divisor $\Gamma$ and special fibre $\mathcal Y^1_0=A\cup \mathcal E^\prime\cup B,$
where $\mathcal E^\prime$ is the blow-up of $\mathcal E$ at $q.$
Still denoting by $F$ the proper transform of $F\subset \mathcal Y$ in $\mathcal Y^1,$ we have that
$F$ has self-intersection $(F)^2_{\mathcal E^\prime}=-1$ on $\mathcal E^\prime.$
Morever,  if $S^1$ is the proper transform of $S$ in $\mathcal Y^1,$ we have 
that 
\begin{equation}\label{step1}
S^1\sim \pi_1^*(S)-3\Gamma.  
\end{equation}
We deduce that $S^1F=-3$ and hence $F\subset S^1.$

{\em Step 2.} Let now $\pi_2:\mathcal Y^2\to\mathcal Y^1$ be the blow-up of $\mathcal Y^1$ along $F$
with new exceptional divisor $\Theta\simeq \mathbb F_1$ and new special fibre 
$\mathcal Y^2_0=A^\prime\cup \mathcal E^\prime\cup\Theta\cup B^\prime,$
where $A^\prime$ and $B^\prime$ are the blow-ups of $A$ and $B$ at $F\cap A$ and $B\cap F$ respectively.
Denoting again by $F$ the proper transform of $ F\subset\mathcal Y^1$ in $\mathcal Y^2,$
we have that $F$ is the $(-1)$-curve of $\Theta.$ Moreover,
if $S^2$ is the proper transform of $S^1$
in $\mathcal Y^2,$ by \eqref{step1}, we deduce that 
\begin{equation}\label{step2}
S^2|_{\Theta}\sim\pi^*_2(S^1)|_{\Theta}-m_F\Theta|_{\Theta}
\sim -3f_\Theta+m_F(F+2f_\Theta)\sim (2m_F-3)f_\Theta+m_FF,
\end{equation}
where $f_\Theta$ is the fibre of $\Theta$ and $m_F$ is the multiplicity of $S^1$ along $F.$
Furthermore, since $S^2|_{\Theta}$ is an effective divisor, we have that $m_F\geq 2.$
In particular, if $m_F=2$ then $|S^2|_{\Theta}|=|f_\Theta+2F|$ contains $F$ in the base locus
with multiplicity $1.$ Hence $S^2|_{\Theta}=F+L,$ with $L\sim f_\Theta+F.$ Using again that
$S^2$ is a Cartier divisor, we find that $S^2|_{A^\prime}$ (resp. $S^2|_{B^\prime}$)
has two smooth branches intersecting $\Theta\cap A^\prime$  
transversally at $F\cap A^\prime$ and $L\cap A^\prime$ (resp. $F\cap B^\prime$ and $L\cap B^\prime$), 
as in Figure 1. 
\begin{figure}[htbp]
\noindent
\begin{minipage}[t]{0.50\textwidth}
\includegraphics[width=6 cm]{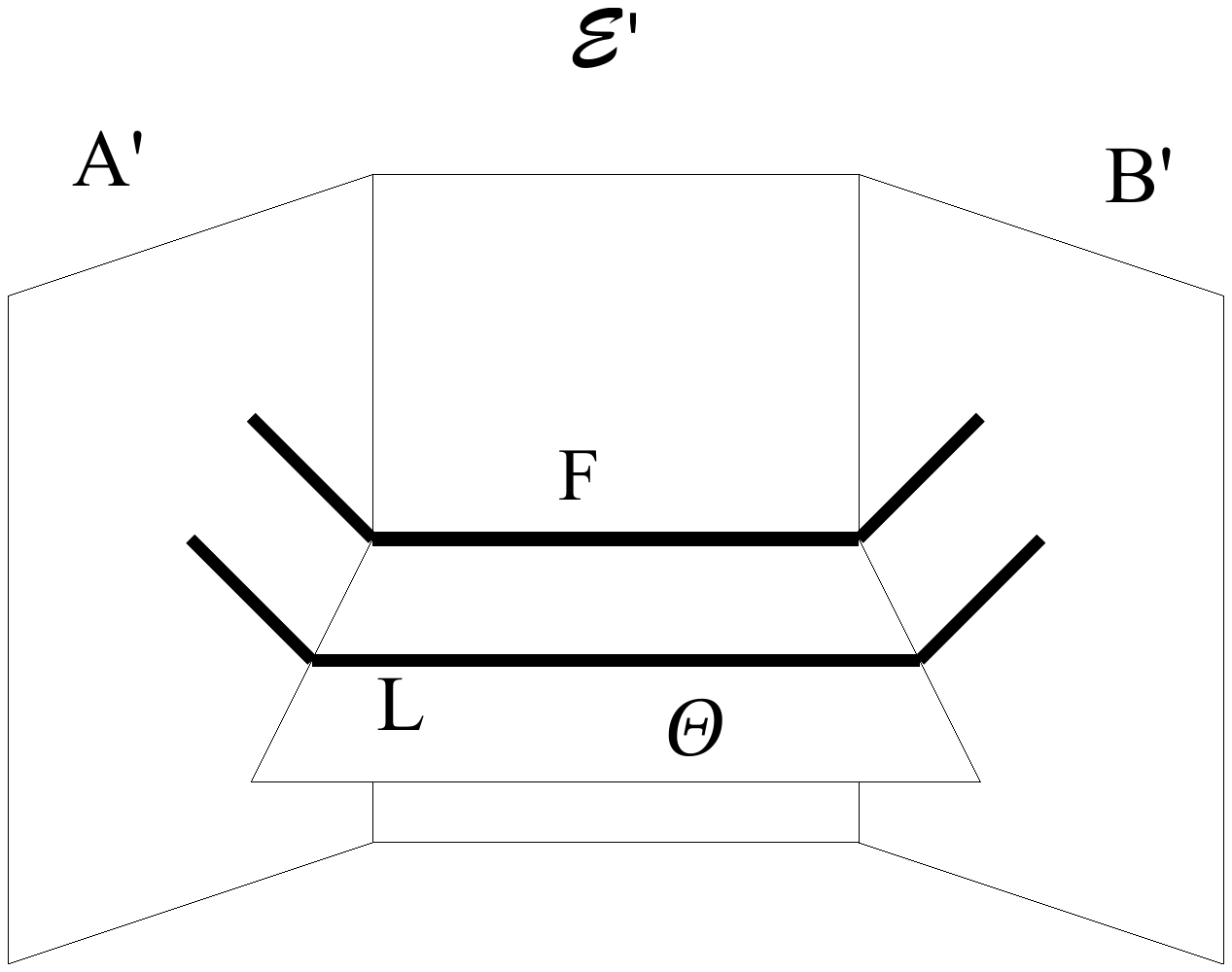}
\caption{}
\end{minipage}
\begin{minipage}[t]{0.47\textwidth}
\includegraphics[width=6 cm]{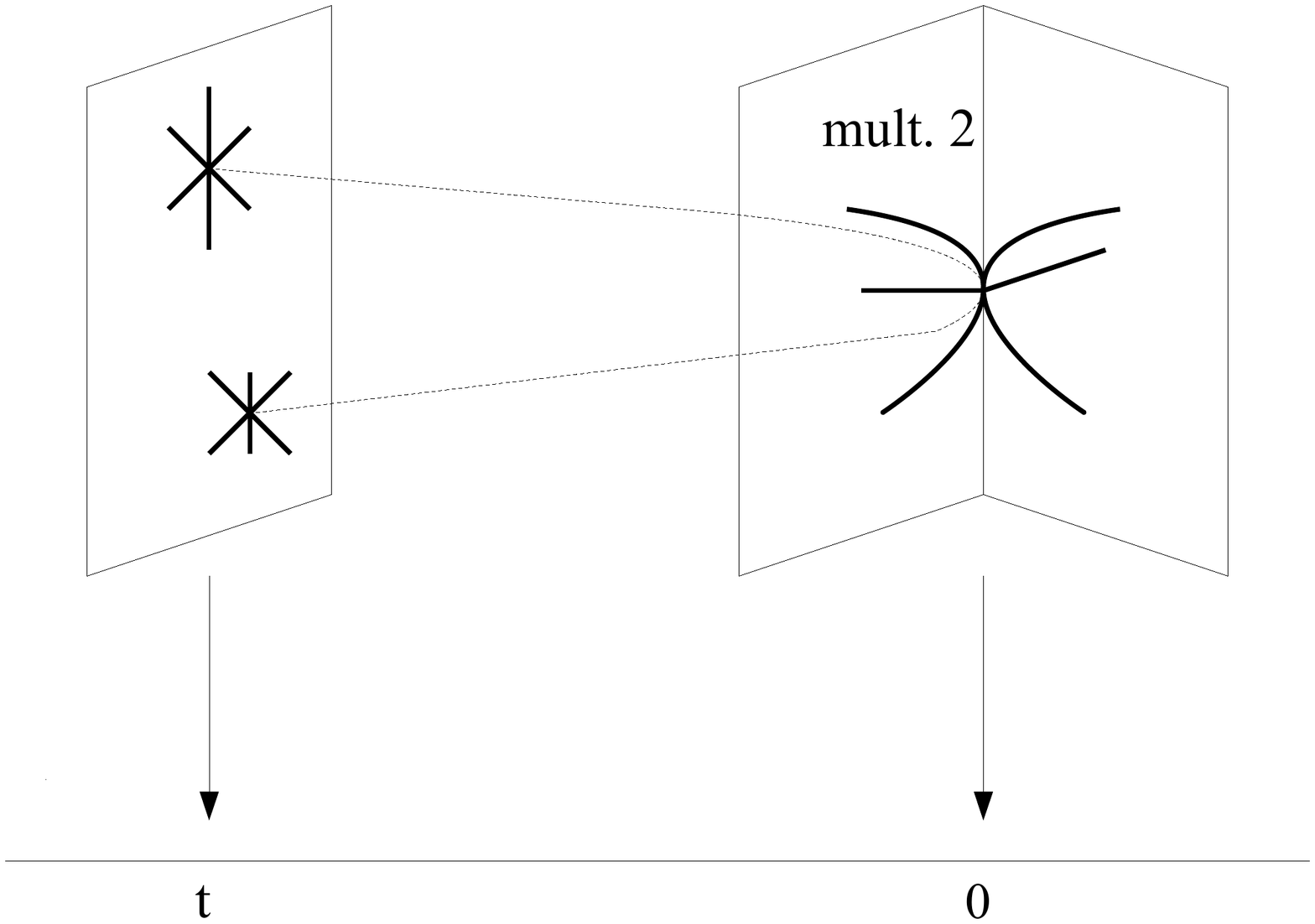}
\caption{}
\end{minipage}\hfill
\end{figure}

\noindent Now let $\mathcal S\subset\mathcal X$ be the image of $S.$ If $t$ is a general point of 
$\mathbb A^1$ and $\{t_1, \,t_2\}=\nu_2^{-1}(t)$, then the fibre of $\mathcal S$ over
$t$ is $\mathcal S_t=S_{t_1}\cup S_{t_2},$ where $S_{t_i}$ is the 
fibre of $S$ over $t_i;$  while the special fibre $\mathcal S_0=2(S|_A\cup S|_B)$ of
$\mathcal S$ is the image curve, counted with multiplicity $2,$   of the special fibre $S_0$ of
$S$ under the contraction of $\mathcal E$ (see Figure 2). 
\begin{remark}
The curve $S|_A\cup S|_B\subset\mathcal X_0$ belongs to the subvariety  $\mathcal Q_p\subset |D_0|,$ defined 
as the Zariski closure of the locus of curves  
$C=C_A\cup C_B,$ where $C_A\subset A$ and $C_B\subset B$ have a node at 
$p\in E=A\cap B$ with one branch tangent to $E.$ 
Notice that every such curve $C$ has a non-planar quadruple point at $p$  and that $\mathcal Q_p$
is a linear subspace of $|D_0|$ of codimension at most $5.$
\end{remark} 

\begin{lemma}\label{lemma: analytic-type}
Let $[C]\in Q_p$ be a point corresponding to a curve $C=C_A\cup C_B$ with a non-planar 
quadruple point at $p$ as in the previous remark. Then the analytic equations
of $C$ at $p$ are given by 
\begin{eqnarray}\label{eq: quadruple-point}
\left\{\begin{array}{l}
(y+x-z^2)z=0\\
xy=t \\
t=0.\\
\end{array}\right.
\end{eqnarray}
\end{lemma}
\begin{remark}
We want to observe that, as for plane curve singularities, the classification of
simple complete intersection space singularities is known, (cf. \cite{giusti}).
The singularity defined by \eqref{eq: quadruple-point} is not a simple singularity. 
\end{remark}
Before proving Lemma \ref{lemma: analytic-type}, we need to recall  a basic result about
space curve singularities. We refer to \cite{giusti} and use the same notation and
terminology.

\begin{definition}
Let  $\mathbb Q_+$ be the set of positive rational numbers.
A polynomial $p(\underline x)\in\mathbb C[x_1,...,x_n]$ is said to be quasi-homogeneous
of type $(d;\,a_1,...,a_n)\in\mathbb Q_+\times\mathbb Q^n_+$ if $p(\underline x)$ is a linear
combination of monomials $x^{\alpha_1}_1\cdots x^{\alpha_n}_n$ such that:
$$\sum_{i=1}^na_i\alpha_i=d.$$
The $n$-tuple  $\underline a=(a_1,...,a_n)$ is called a system of weights.
We shall also say that $p(\underline x)$ has degree $d$ if every variable $x_i$ has weight $a_i,$
for every $i\leq n.$ 
\end{definition}
For every fixed system of weights $\underline a=(a_1,...,a_n),$ we will denote by $\Gamma^{\underline a}$
the induced graduation
$$\mathbb C[x_1,...,x_n]=\oplus_{d\geq 0}G_d^{\underline a}$$ 
on $\mathbb C[x_1,...,x_n],$ where $G_d^{\underline a}$ is the set of quasi-homogeneous polynomials
of type $(d;\,a_1,...,a_n).$
\begin{definition}
An element $f=(f_1,...,f_p)\in (C[x_1,...,x_n])^p$ is quasi-homogeneous of type 
$(\underline d,\underline a)=(d_1,...,d_p;\,a_1,...,a_n)$ if, for every $i=1,...,p$, the component 
$f_i$ of $f$ is of type $(d_i;\,a_1,...,a_n).$ The $(p+n)$-tuple $(\underline d,\underline a)=(d_1,...,d_p;\,a_1,...,a_n)$
is called a system of degrees and weights.
\end{definition}
For every fixed system of degrees and weights $(\underline d,\underline a),$ one defines a graduation 
$\Gamma^{(\underline d,\underline a)}$ on $ (C[x_1,...,x_n])^p$ by setting 
$$(C[x_1,...,x_n])^p=\oplus_{\nu\in\mathbb Z}G^{(\underline d,\underline a)}_\nu,$$
where $G^{(\underline d,\underline a)}_\nu=\{(g_1,...,g_p)\in (C[x_1,...,x_n])^p|\,g_i\in 
G^{\underline a}_{d_i+\nu}\}.$
Moreover, we denote by $\nu_{\underline a}$ the valuation naturally associated to 
$\Gamma^{\underline a}$ and by $\nu_{\underline d,\underline a}$ the valuation associated
to $\Gamma^{\underline a,\underline d}$ and defined by 
$$
\nu_{\underline d,\underline a}(h)=\inf_i[\nu_{\underline a}(h_i)-d_i],
$$
for every $h\in (C[x_1,...,x_n])^p.$
Finally, we recall that, if $I_{n,p}$ denotes the set of germs of applications 
$f=(f_1,...,f_p):\mathbb (C^n,\underline 0)\to\mathbb (C^p,\underline 0)$ such $(X,\underline 0)=
(f^{-1}(\underline 0),\underline 0)$
is a germ of a complete intersection analytic variety with an isolated singularity at the origin,
then we have the versal deformation space $T^1(f)=T^1_{X,\underline 0}$ (cf. \cite{GLS} or \cite{ser})  
and,  by \cite{Pi}, a graduation is naturally induced on 
$$
T^1(f)=\oplus_{\nu\in\mathbb Z}T^1(f)_\nu, \textrm{where}\,\,\,T^1(f)_\nu\subset G^{(\underline d,\underline a)}_\nu
\textrm{for every}\,\,\,\nu\in\mathbb Z.
$$
\begin{proposition}[Merle, {\cite[Proposition 1]{giusti}}]\label{merle}
Let $f\in I_{n,p}$ be a quasi-homogeneous element of type $(\underline d,\underline a)$.
Let $g\in (\mathbb C[x_1,...,x_n])^p$ be any element such that:
$$
\nu_{(\underline d,\underline a)}(g)>\sup(0,\alpha),
$$
where $\alpha=Sup_{\nu\in\mathbb Z}\{\nu | T^1(f)_\nu\neq 0\}.$ Then the germs of singularities
$((f+g)^{-1}(\underline 0),\underline 0)$ and $(f^{-1}(\underline 0),\underline 0)$ are analytically equivalent. 
\end{proposition}

\begin{proof}[Proof of Lemma \ref{lemma: analytic-type}]
Let $[C^\prime]\in Q_p$ be a point associated to a curve $C^\prime=C_A^\prime\cup C_B^\prime,$  where
 $C_A^\prime\subset A$ 
and $C_B^\prime\subset B$ have a node at $p\in E$ with one branch tangent to $E=A\cap B.$ Let
$(x,y,z,t)$ be analytic coordinates of $\mathcal X$ at $p=\underline 0$ in such a way that $A:x=t=0$ and $B:y=t=0.$ 
Then the analytic equations of $C^\prime\subset\mathcal X$ at $p$ are given by 
\begin{eqnarray}\label{eq: general-equation}
\left\{\begin{array}{l}
p(x,y,z)=0,\\
xy=0, \\
t=0,
\end{array}\right.
\end{eqnarray}
where $p(x,y,z)=0$ is the equation of an analytic surface $S$ in $\mathbb A^3$ having 
a singularity of multiplicity $2$ at $(0,0,0).$ Moreover, the tangent cone of $S$ at $p$ must contain the
line $\,x=y=0.$ Thus, up to a linear transformation, we may assume that the 
analytic equations of $C^\prime\subset\mathcal X$ at $p$ are given by
$$
p(x,y,z)=xz+yz+ax^2+by^2+z^3+p_3(x,y,z)=0,\,\,xy=0,\,\,t=0,
$$
where $p_3(x,y,z)$ is a polynomial of degree at least $3$ with no $z^3$-term. 
We want to prove that the complete intersection curve singularity given by \eqref{eq: general-equation}
is analytically equivalent to the curve singularity given by the equations \eqref{eq: quadruple-point}.
We first observe
that the element $f=(xz+yz+z^3,xy)\in I_{3,2}$ is quasi-homogeneous of type
$(3,4;\,2,2,1).$ Moreover, every term of the polynomial $p(x,y,z)-xz-yz-z^3=ax^2+by^2+p_3(x,y,z)$ 
has degree strictly greater than $3,$ if the variables $(x,y,z)$ have weights $(2,2,1).$ In order to apply
Proposition \ref{merle}, we need to compute $T^1(f).$ Let $C$ be a reduced Cartier divisor on $\mathcal X_0=A\cup B$ 
having a singularity
of analytical equations given by \eqref{eq: quadruple-point} at $p$. Then, since $C$ is a local complete intersection subvariety
of $\mathcal X,$ we have the following standard exact sequence of sheaves on $C$
 \begin{equation}\label{standardsequence}
\xymatrix{
0  \ar[r] &      \Theta_C  \ar[r] &      \Theta_{\mathcal X}|_C \ar[r]^{\alpha} & \mathcal N_{C|\mathcal X} \ar[r] &   T^1_C \ar[r] &     0, }
\end{equation}
where $\Theta_C\simeq Hom(\Omega_C^1,\mathcal O_C)$ is the tangent sheaf of $C,$ defined 
as the dual of the sheaf of differentials of $C,$ $\Theta_{\mathcal X}|_C$ is the tangent sheaf of 
$\mathcal X$ restricted to $C,$ $\mathcal N_{C|\mathcal X}$ is the normal 
bundle of $C$ in $\mathcal X$
and  $T^1_C$ is the first cotangent sheaf of $C$, which is supported at the singular 
locus $\Sing(C)$ of $C$ and whose stalk $T^1_{C,q}$ at every singular point
$q$ of $C$  is the versal deformation space of the singularity.
We also recall that the global sections of the image sheaf  $\mathcal N_{C|\mathcal X}^\prime$ of $\alpha$ 
are, by the versality properties of $T^1_C,$ the infinitesimal deformations of $C$ in $\mathcal X$ preserving
singularities of $C$ and their analytic type. For this reason, $\mathcal N_{C|\mathcal X}^\prime$ is called
the {\em equisingular deformation sheaf of $C$ in $\mathcal X.$}
Now observe that $T^1(f)=T^1_{C,p}.$
In order to compute $T^1_{C,p}$ we use
the following standard identifications:
 \begin{itemize}
 \item of the local ring $\mathcal O_{C,\,p}=\mathcal O_{\mathcal X,p}/\mathcal I_{C|\mathcal X,p}$ of $C$ at $p$ 
 with $\mathbb C[x,y,z]/(f_1,f_2)$, where $f_1(x,y,z)=xz+yz+z^3$ 
 and $f_2(x,y,z)=xy,$ \\
 \item of the $\mathcal O_{C,p}$-module $N_{C|\mathcal X,\,p}$ with the free $\mathcal O_{\mathcal X,\,p}$-module 
 $Hom_{\mathcal O_{\mathcal X,\,p}}(\mathcal I_{C|\mathcal X ,\,p}, \mathcal O_{C,p}),$ generated by
 morphisms $f_1^*$ and $f_2^*$,  defined by 
 $$f_i^*(s_1(x,y,z)f_1(x,y,z)+s_2(x,y,z)f_2(x,y,z))=s_i(x,y,z),\,\mbox{for}\, i=1,\,2$$
 \end{itemize}
 and, finally,
 \begin{itemize}
 \item of the $\mathcal O_{C,p}$-module
 \begin{eqnarray*}
 (\Theta_{\mathcal X}|_C)_{\, p}&\simeq &\Theta_{\mathcal X,p}/(I_{C,p}\otimes\Theta_{\mathcal X,p})\\
& \simeq & \langle \partial /{\partial x},\partial /{\partial y},\partial / {\partial z},\partial /{\partial t} \rangle
 _{\mathcal O_{C,\,p}}/
 \langle {\partial}/{\partial t}-x\partial /{\partial y}-y\partial /{\partial x}\rangle
 \end{eqnarray*}
 with the free $\mathcal O_{\mathcal X,\,p}$-module generated by the derivatives 
 $\partial /{\partial x},\partial / {\partial y},\partial /{\partial z}.$ 
 \end{itemize}
With these identifications, 
the localization map $\alpha_p:(\Theta_{\mathcal X}|_C)_{\, p}\rightarrow \mathcal N_{C|\mathcal X,\,p}$
at $p$ 
of the sheaf map $\alpha$ in \eqref{standardsequence}  is defined by
 \begin{eqnarray*}
\alpha_p(\partial /{\partial x})&=&\,\,\,\,\,\,\Big(s=s_1f_1+s_2f_2\rightarrowtail 
\partial s/\partial x=_{\mathcal O_{C,p}}
s_1\partial f_1 /{\partial x}
+s_2\partial f_2 /{\partial x}\Big)
\end{eqnarray*}
and, similarly, for $\alpha_p(\partial /{\partial y})$ and $\alpha_p(\partial /{\partial z}).$
In particular, we have that
 \begin{eqnarray}\label{components}
 \begin{array}{cccc}
\alpha_p(\partial /{\partial x})&=& zf_1^*(s)\,\,\,+\,\,yf_2^*(s),&\\
\alpha_p(\partial /{\partial y})&=& zf^*_1(s)\,\,\,+\,\,xf^*_2(s)\,&\textrm{and}\\
\alpha_p(\partial /{\partial z})&=& (x+y+3z^{2})f^*_1(s).&
\end{array}
\end{eqnarray}
It follows that the versal deformation space of the non planar quadruple point of $C$ at $p$ is the 
affine space
\begin{equation}
T^1_{C,p}\simeq\mathcal O_{C,p}f_1^*\oplus \mathcal O_{C,p}f_2^*/_{\langle zf_1^*+yf_2^*,
zf_1^*+xf_2^*,(3z^2+x+y)f_2^*\rangle}.
\end{equation}
In particular, $T^1_{C,p}=T^1(f)\simeq\mathbb C^7$
(in accordance with \cite[Proposition on p. 165]{giusti1}) 
and, if we fix affine coordinates 
$(b_1,\,b_2,\,b_3,\,a_1,\,a_2,\,a_3,\,a_4)$ on $T^1_{C,p},$ then
the versal deformation family $\mathcal C_p\to T^1_{C,p}=T^1(f)$ has equations
\begin{eqnarray}\label{versal_family}
\left\{\begin{array}{l}
(y+x-z^2)z+a_1+a_2x+a_3y+a_4z=F(x,y,z)=0,\\
xy+b_1+b_2z+b_3z^2=G(x,y,z)=0. \\
\end{array}\right.
\end{eqnarray}
Furthermore, by the equality
$$
\begin{pmatrix}F(x,y,z)\\ G(x,y,z)
\end{pmatrix}=\begin{pmatrix}0\\b_1
\end{pmatrix}+\begin{pmatrix}a_1\\ b_2z
\end{pmatrix}+
\begin{pmatrix}a_4z\\b_3z^2
\end{pmatrix}+\begin{pmatrix}a_3y+a_2x\\ 0
\end{pmatrix}+\begin{pmatrix}(y+x-z^2)z\\ xy
\end{pmatrix},
$$ we deduce that the graduation on $T^1(f)$
 induced by $\Gamma^{(3,4;\,2,2,1)}$ is given by
 $$
 T^1(f)=\oplus_{\nu=-4}^0 T^1(f)_\nu.
 $$ 
In particular, we obtain that $\alpha=\sup_{\nu\in\mathbb Z}\{\nu | T^1(f)_\nu\neq 0\}=0.$
 By Proposition \ref{merle}, using that $\nu_{3,4;\,2,2,1}(p(x,y,z)-xz-yz-z^3, xy)>3>0=\alpha$
 we get that the singularity defined by \eqref{eq: general-equation} is analytically 
 equivalent to the singularity of equations \eqref{eq: quadruple-point}, as
 desired. 
\end{proof}

\begin{remark}[\bf{Equisingular infinitesimal  deformations}]\label{eq-eq-in}
 Let $C\subset\mathcal X_0$ be a curve as in Lemma \ref{lemma: analytic-type}.
 By the proof of the lemma, we may deduce the equations at $p$ of an equisingular
 infinitesimal deformation of $C$ in $\mathcal X.$ Indeed, by 
  the equalities \eqref{components}, we have that, if $s\in \mathcal N_{C|\mathcal X,p}^\prime,$ 
then the equations of $s$ at $p$ are given by 
\begin{eqnarray}\label{equisingulardeformationquad}
\left\{\begin{array}{l}
z(x+y+z^2)+\epsilon (zu_x+zu_y+(3z^{2}+x+y)u_z)=0,\\
xy+\epsilon (xu_y+yu_x) =0,
\end{array}\right.
\end{eqnarray}
where $$u=u_x(x,y,z)\partial /{\partial x}+u_y(x,y,z)\partial /{\partial y}+
u_z(x,y,z)\partial /{\partial z}\in {\Theta_{\mathcal X}|_D}_{\, p}.$$
Now $xy+\epsilon (xu_y+yu_x)=0$ is the equation of an infinitesimal deformation 
$\mathcal X_0$ vanishing along the singular locus $E.$ By  \cite[Section 2]{chm}, we know that these 
infinitesimal deformations are the infinitesimal deformations of $\mathcal X_0$ preserving the
singular locus. Since $\mathcal X_0$ is the only singular fibre
of $\mathcal X$, we have that  $xu_y+yu_x=0$ in \eqref{equisingulardeformationquad}.
Since this does not depend on the type of singularity of $C$ at $p,$ we deduce that
$H^0(C,\mathcal N_{C|\mathcal X}^\prime)=H^0(C,\mathcal N^\prime_{C|\mathcal X_0}).$ Moreover,
we obtain that $u_x$ and $u_y$ are polynomials with no $z^r$-terms, for every $r,$
and such that $u_y(\underline 0)=u_x(\underline 0)=0$ and
no $x^r$-terms (resp. $y^r$-terms) appear in $u_y$ (resp. $u_x$).  
Thus the equations of $s\in \mathcal N_{C|\mathcal X,p}^\prime$ at $p$ are given by
\begin{eqnarray}\label{equisingulardeformationquad2}
\left\{\begin{array}{l}
z(x+y+z^2)+\epsilon (zxu^\prime_x+zyu^\prime_y+(3z^{2}+x+y)(c+u^\prime_z))=0,\\
xy =0,
\end{array}\right.
\end{eqnarray}
where we set $u_x(x,z)=xu^\prime_x(x,z),$ $u_y(y,z)=yu^\prime_y(y,z)$ and
$u_z(x,y,z)=c+u^\prime_z(x,y,z),$ with $u^\prime_z(\underline 0)=0.$
Finally, $H^0(C,\mathcal N^\prime_{C|\mathcal X})=H^0(C,\mathcal N^\prime_{C|\mathcal X_0})$ is
a linear space of codimension $\leq 4$ in $|D_0|,$ contained in the linear system
$T_p$ of curves $C^\prime_A\cup C^\prime_B\subset\mathcal X_0,$ with $C^\prime_A
\subset A,$ $C^\prime_B\subset B$ and $C_A^\prime$ and $C_B^\prime$ tangent to $E$ at $p.$
\end{remark}

We may now provide sufficient conditions for the existence 
of curves with a triple point and possibly further singularities in the linear system  $|D_t|.$

\begin{definition}
Let $\mathcal X$  be a family of surfaces as above and let $D\subset \mathcal X$ be a Cartier divisor.
Let $\mathcal W^{\mathcal X|\mathbb A^1}_{D,g,tr} \subset |\mathcal
O_{\mathcal X}(D)|\times (\mathbb A^1\setminus\{0\})$ be the locally closed subset defined as follows
\begin{eqnarray*}
\mathcal W^{\mathcal X|\mathbb A^1}_{D,g,tr} =\{([D^\prime],t)\mid\,\mathcal
X_t\,\textrm{is smooth,}\, D^\prime\cap \mathcal X_t:=D_t^\prime\in\,\mid\mathcal
O_{ \mathcal X_t}(D_t)\mid\, \textrm{is
irreducible}\\
\textrm{of genus}\,g,\,\textrm{with a triple point and nodes as singularities}.\}
\end{eqnarray*}
There is a naturally defined rational map $\pi:\mathcal W^{\mathcal X|\mathbb A^1}_{D,g,tr} \to\mathcal H^{\mathcal X|\mathbb A^1},$
where $\mathcal H^{\mathcal X|\mathbb A^1}$ is the relative Hilbert scheme of $\mathcal X\to\mathbb A^1.$
We will denote by $\mathcal V^{\mathcal X|\mathbb A^1}_{D,g,tr}$ the Zariski closure of the image of $\pi$ 
and we will name it the universal Severi variety
of curves of genus $g$ in $|D|$ with a triple point and nodes. The restriction of this variety 
$\mathcal V^{\mathcal X|\mathbb A^1}_{D,g,tr}\cap |D_t|=\mathcal V^{\mathcal X_t}_{D_t,g,tr},$
where $t\in\mathbb A^1$ is general, is the Severi variety of genus $g$ curves in $|D_t|$
with a triple point and nodes as singularities.
\end{definition}
We observe that, if $V\subset \mathcal V^{\mathcal X_t}_{D_t,g,tr} $ is an irreducible component, then 
 $V$ coincides with
the equisingular deformation locus $ES(C)\subset |D_t|$ of the curve $C$ corresponding to the general point
$[C]\in V,$ defined in Section \ref{notation-terminology}. 
\begin{theorem}\label{th: triple-point}
With the notation above, let $C=C_A\cup C_B\subset \mathcal X_0$ be a reduced divisor 
in the linear system $|D_0|$ such that $C_A\subset A$ and $C_B\subset B$ have a node at a general point $p$ of
$E=A\cap B$ with one branch tangent to $E.$ Suppose that $C_A$ and $C_B$ are smooth 
 and they intersect $E$ transversally outside $p.$
Assume, moreover, that:
\begin{itemize}
\item[1)] $h^1(A,\mathcal O_A)=h^1(B,\mathcal O_B)=h^1(\mathcal X_t, \mathcal O_{\mathcal X_t})=0,$ for every $t;$
\item[2)] $\dim(|D_t|)=\dim(|D_0|),$ for a general $t;$\label{liscezza_schema_hilbert} 
\item[3)] $h^0(C,\mathcal N^\prime_{C|\mathcal X})=\dim(|D_0|)-4.$\label{expected_dimension}
\end{itemize}
Then the universal Severi variety $\mathcal V^{\mathcal X|\mathbb A^1}_{D,p_a(D)-3,tr}
\subset\mathcal H^{\mathcal X|\mathbb A^1}$ is non-empty. More precisely, the point $[C]\in\mathcal H^{\mathcal X|\mathbb A^1}$
corresponding to the curve $C$ belongs to an irreducible component $\mathcal Q$
of the special fibre $\mathcal V_0$ of $\mathcal V^{\mathcal X|\mathbb A^1}_{D,p_a(D)-3,tr}\to\mathbb A^1,$ contained in $\mathcal V_0$
with multiplicity $2.$ Finally, $\mathcal V^{\mathcal X|\mathbb A^1}_{D,p_a(D)-3,tr}$ is smooth at $[C]$
and the irreducible component $V_t^{\mathcal Q}$ of the general fibre $\mathcal V_t$ of $\mathcal V,$
specializing to $\mathcal Q,$ has expected dimension in $|D_t|.$
\end{theorem}
\begin{proof}
We want to obtain curves in the linear system $|\mathcal O_{\mathcal X_t}(D_t)|$ with a triple point 
as deformations of $C\subset\mathcal X_0.$
The scheme parametrizing deformations of $C$ in $\mathcal X$ is an irreducible component 
$\mathcal H$ of the relative Hilbert scheme $\mathcal H^{\mathcal X|\mathbb A^1}$ of the family $\mathcal X.$ 
In particular, by \cite[Proposition 4.4.7]{ser} and the hypotheses $1)$ and $2)$, we have that
$\mathcal H$ is smooth at the point $[C]\in\mathcal H$
corresponding to $C.$  Now, by hypothesis, $C$ has a non-planar quadruple point at $p$,
nodes on $E\setminus \{p\}$ and no further singularities. Moreover,
 it is well-known that, no matter how we deform
$C$ to a curve on $\mathcal X_t,$ the nodes of $C$ on $E$ are smoothed (see, for example, \cite[Section 2]{galati}).
We want to prove that $C$ may be deformed to a curve on $\mathcal X_t$ in such a way that
the non-planar quadruple point of $C$ at $p$ is deformed to a 
triple point.  This will follows from a local analysis.
First recall that, by Lemma \ref{lemma: analytic-type}, we may choose analytic coordinates
$(x,y,z,t)$ of $\mathcal X$ at $p$ in such a way that the equations of $C$ at $p$ are given by 
\eqref{eq: quadruple-point} and the versal deformation family $\mathcal C_p\to T^1_{C,p}$
has equations given by \eqref{versal_family}, where $(b_1,\,b_2,\,b_3,\,a_1,\,a_2,\,a_3,\,a_4)$ 
are the affine coordinates on $T^1_{C,p}\simeq\mathbb C^7.$ By versality,
denoting by $\mathcal D\to\mathcal H$ the universal family parametrized by $\mathcal H,$
there exist \'etale neighborhoods
$U_p$ of $[C]$ in $\mathcal H,$ $U_p^\prime$ of $p$ in $\mathcal D$ and  $V_p$ of $\underline 0$ in $T^1_{C,q}$
 and a map $\phi_p:U_p\to V_p$ 
so that the family
$\mathcal D|_{U_p}\cap U_p^\prime$ is isomorphic to the pull-back of $\mathcal C_p|_{V_p}$, with respect to $\phi_p$,
\begin{equation}\label{versalmap}
\xymatrix{
\mathcal C_p  \ar[d] & \mathcal C_p |_{V_p} \ar[d]  \ar[l]& 
U_p\times_{V_p}\mathcal C_p |_{V_p}  \ar[l] \ar[r]^{\hspace{0.5cm}\simeq} \ar[dr] & \mathcal D|_{U_p}\cap U_p^\prime
\ar[r] \ar[d] & \mathcal D  \ar[d] \\
T^1_{C,p} &  V_p \ar[l] & & \ar[ll]_{\phi_p} U_p \ar[r] & \mathcal H.}
\end{equation}  
We need to describe the image $\phi_p(U_p)\subset V_p.$ First we want to prove that
\begin{equation}\label{special-map}
\phi_p(U_p\cap |D_0|)=\Gamma\cap V_p,
\end{equation} 
where $$\Gamma:  b_1= b_2 = b_3=0.$$
Obviously, $\phi_p(U_p\cap |D_0|)\subset\Gamma\cap V_p.$ To prove that
the equality holds, it is enough to show that the differential 
${d\phi_p}_{[C]}:H^0(C,\mathcal N_{C|\mathcal X_0})\to T_{\underline 0}\Gamma$ is surjective. 
By standard deformation theory, by identifying the versal deformation space
of a singularity with its tangent space at $\underline 0,$ the differential 
of $\phi_p$  at $[C]$ can be identified with the map
 $$
{ d\phi_p}_{[C]}:H^0(C,\mathcal N_{C|S})\to H^0(C, T^1_C)\to T^1_{C,p},
 $$
 induced by the exact sequence \eqref{standardsequence}.
In particular, using that, by  Remark \ref{eq-eq-in}, $H^0(C,\mathcal N^\prime_{C|\mathcal X})=
H^0(C,\mathcal N^\prime_{C|\mathcal X_0})$ and that the nodes of $C$ on $E$ are 
necessarily preserved when we deform $C$ on $\mathcal X_0,$
we have that $$\ker{d\phi_p}_{[C]}=H^0(C,\mathcal N^\prime_{C|\mathcal X}).$$
Now, again by Remark \ref{eq-eq-in}, we know that $h^0(C,\mathcal N^\prime_{C|\mathcal X})=
h^0(C,\mathcal N^\prime_{C|\mathcal X_0})\geq \dim(|D_0|)-4=\dim(|D_0|)-\dim(\Gamma).$ 
Actually, by the hypothesis $3),$
we have that $h^0(C,\mathcal N^\prime_{C|\mathcal X_0})$ has the expected dimension
and hence the equality \eqref{special-map} is verified.
Using that $H^0(C,\mathcal N_{C|\mathcal X_0})$
is a hyperplane in $H^0(C,\mathcal N_{C|\mathcal X})$ and that $\phi_p(U_p\cap |D_0|)
\subset\phi_p(U_p),$ this implies, in particular, that $\phi_p(U_p)\subset V_p\subset T^1_{C,p}$ is a subvariety
of dimension $\dim(\phi_p(U_p\cap |D_0|))+1=5,$ smooth at $\underline 0.$ We want to determine 
the equations of tangent space $T_{\underline 0}\phi_p(U_p)=
d\phi_p(H^0(C,\mathcal N_{C|\mathcal X})).$ For this purpose, it is enough to find 
 the image by
$d\phi_p$ of the infinitesimal deformation $\sigma\in H^0(C,\mathcal N_{C|\mathcal X})\setminus
H^0(C,\mathcal N_{C|\mathcal X_0})$ having equations
\begin{eqnarray}
\left\{\begin{array}{l}
z(x+y+z^2)=0\\
xy=\epsilon.
\end{array}\right.
\end{eqnarray}  
The image of $\sigma$ is trivially the vector $(1,0,...,0).$
We deduce that 
$$
T_{\underline 0}\phi_p(U_p)=
d\phi_p(H^0(C,\mathcal N_{C|\mathcal X})):\,b_2=b_3=0.
$$

Now we want to prove that the locus in $\phi_p(U_p)\setminus \phi_p(U_p\cap |D_0|)$ of points
 corresponding to curves with an ordinary triple point as singularity is not
empty and its Zariski closure is a smooth curve $ T$ tangent to 
$\Gamma=\phi_p(U_p\cap |D_0|)$ at $\underline 0.$
This will imply the theorem by versality and by a straightforward dimension count. 
Because of smoothness of $\phi_p(U_p)$ at $\underline 0,$ it is enough to prove that the locus
of points $(b_1,0,0,a_1,...,a_4)\in T_{\underline 0}\phi_p(U_p)$ with $b_1\neq 0$ and  corresponding
to a curve with an ordinary triple point is not empty  and its Zariski closure is a smooth curve $ T$ tangent to 
$\Gamma=\phi_p(U_p\cap |D_0|)$ at $\underline 0.$

Let $(b_1,0,0,a_1,...,a_4)\in T^1_{C,p}$ be a point with $b_1\neq 0.$ Then the fibre $\mathcal C_{(b_1,0,0,a_1,...,a_4)}$
of the versal family $\mathcal C_p\to T^1_{C,p}$ has equation
$$
\mathcal C_{(b_1,0,0,a_1,...,a_4)}: z(z^2+\frac{b_1}{y}+y)+a_1+a_2\frac{b_1}{y}+a_3y+a_4z=0, \,\,y\neq 0
$$
or, equivalently,
$$
\mathcal C_{(b_1,0,0,a_1,...,a_4)}: F(y,z)=z^3y+zy^2+b_1z+a_1y+a_2b_1+a_3y^2+a_4zy, \,\,y\neq 0.
$$
Now a point $(y_0,z_0)\in\mathcal C_{(b_1,0,0,a_1,...,a_4)}$
is a singular point of multiplicity at least three if and only if
\begin{eqnarray}
\frac{\partial F(y,z)}{\partial y}&=& z_0^3+2y_0z_0+a_1+2a_3y_0+a_4z_0=0,\label{1}\\
\frac{\partial F(y,z)}{\partial z}&=& 3z_0^2y_0+y^2_0+b_1+a_4y_0=0,\label{2}\\
\frac{\partial F(y,z)}{\partial z^2}&=& 6z_0y_0=0,\label{3}\\
\frac{\partial F(y,z)}{\partial y^2}&=& 2z_0+2a_3=0,\label{4}\\
\frac{\partial F(y,z)}{\partial y\partial z}&=& 3z_0^2+2y_0+a_4=0.\label{5}
\end{eqnarray}
By the hypothesis $b_1\neq 0$ and the equalities \eqref{3} and \eqref{4}, we find that $z_0=0.$
By substituting $z_0=0$ in the equalities \eqref{1}, \eqref{2}, \eqref{4}, \eqref{5} and $F(y_0,0)=0,$
we find that  
\begin{eqnarray*}
a_1=a_2=a_3=0,\\
y^2_0+b_1+a_4y_0=0,\\
z_0=0,\\
2y_0+a_4=0.
\end{eqnarray*}
Conversely, for every point $(b_1,0,0,a_1,...,a_4)\in T^1_{C,p}$ such that $b_1\neq 0,$
$a_1=a_2=a_3=0$ and $a_4^2=4b_1,$ we have that the corresponding curve has 
a triple point at $(-\frac{a_4}{2},0),$ with tangent cone of equation 
$z((y+\frac{a_4}{2})^2+z^2)$ and no further singularities. The curve  
$$
T: a_1=a_2=a_3=b_2=b_3=0, a_4^2=4b_1
$$
is smooth and tangent to  $\Gamma: b_1=b_2=b_3=0$ at $\underline 0.$
\end{proof}
The following corollary is a straight consequence of the proof of Theorem \ref{th: triple-point}.

\begin{corollary}\label{ingenerale}
Let $\mathcal X$  be a family of regular surfaces and $D\subset \mathcal X$ a Cartier divisor
as in the statement of Theorem \ref{th: triple-point}. 
Let $ C^\prime\subset |D_0|$ be any reduced curve with a space quadruple 
point of equations \eqref{eq: quadruple-point} at a
point $p\in E$ and possibly further singularities. Then, using the same notation as in the proof
of the previous theorem, the image ${H_p}^\prime$ of the morphism
\begin{equation*}
\xymatrix{
H^0(C^\prime,\mathcal N_{C^\prime|\mathcal X})\ar[r]&T^1_{C^\prime,p}}
\end{equation*}
is contained in the $5$-dimensional plane of equations $H_p:b_2=b_3=0.$ If ${H_p}^\prime=H_p$
then there exist deformations $C_t\in |D_t|$ of $C$ on $\mathcal X_t$ with a triple point, obtained 
as deformation of the singularity of $C$ at $p.$ 
\end{corollary}

\section{Curves with a triple point and nodes on general $K3$ surfaces}
This section is devoted to the proof of Theorem \ref{main-theorem}.
We will prove the theorem by using the very classical  degeneration technique introduced
in \cite{clm}. Let $(S,H)$ be a general primitively polarized $K3$ surface of genus 
${\p}=p_a(H)$ in $\mathbb P^{\p}.$
We will degenerate $S$ to a union of two rational normal scrolls $R=R_1\cup R_2.$
On $R$ we will prove the existence of suitable curves $C\in |\mathcal O_{R}(nH)|$
with a space quadruple point given by equations \eqref{eq: quadruple-point}, tacnodes and nodes. Finally, we will
show that the curves $C$ deform to curves on $S$ with the desired singularities.

We first explain the degeneration argument, introducing notation. 
 Fix an integer $\p \geq 3$ and set $l:= \lfloor \frac{\p}{2} \rfloor$.
Let $E \subset \mathbb P^{\p}$ be a smooth elliptic normal curve
of degree $\p+1$. Consider two general line bundles 
$L_1, L_2 \in \Pic^2(E) $.
We denote by $R_1$ and $R_2$ the (unique) rational normal scrolls of degree $\p-1$ in
$\mathbb P^{p}$ defined by $L_1$ and $L_2$, respectively. Notice that $R_1 \cong R_2 
\cong \mathbb P^1 \times \mathbb P^1\cong \mathbb F_0$ if $\p$ is odd 
whereas $R_1 \cong R_2 \cong \mathbb{F}_1$ if $\p$ is even. Moreover, $R_1$ and  $R_2$ intersect 
transversally along the curve $ E$ which  
 is anticanonical in each $R_i$ (cf. \cite[Lemma 1]{clm}).
More precisely, for odd $\p$, where $R_1 \cong R_2 
\cong \mathbb P^1 \times \mathbb P^1$, we let $\sigma_i=\mathbb P^1 \times \{pt\}$ and $F_i=\{pt\} \times \mathbb P^1$ on $R_i$ be the generators of $\Pic R_i$, with $i=1,2$.
For even $\p$, where $R_1 \cong R_2 \cong \mathbb{F}_1$, we let $\sigma_i$ be the section of negative self-intersection and $F_i$ be the class of a fiber. Then the embedding of $R_i$ into $\mathbb P^{\p}$ is given by the line bundle 
$\sigma_i+lF_i,$ for $i=1,2$.
Let now $R:= R_1 \cup R_2$ and let $\mathcal U_{\p}$ be the component of the Hilbert scheme of $\mathbb P^{\p}$ containing $R$. Then we have that $\dim(\mathcal U_{\p})={\p}^2+2{\p}+19$ and, by  \cite[Theorems 1 and 2]{clm}, 
 the general point $[S]\in\mathcal U_{\p}$ represents a smooth, projective $K3$ surface $S$ of degree $2{\p}-2$ in $\mathbb P^{\p}$ such that $\Pic S \cong \mathbb Z[\mathcal O_S(1)]=\mathbb Z[H].$ 
We denote by $\mathcal{S} \to T$ a general deformation of $\mathcal S_0=R$ over a
one-dimensional disc $T$ contained in $\mathcal U_{\p}$. In particular, the general fiber
is a smooth projective  $K3$ surface $\mathcal S_t$ in $\mathbb P^{p}$ with $\Pic \mathcal S_t \cong \mathbb Z [\mathcal O _{\mathcal S_t}(1)]$.
Now $\mathcal{S}$ is smooth except for $16$ rational double points $\{\xi_1, \ldots, \xi_{16}\}$ lying on
$E.$ In particular, $\{\xi_1, \ldots, \xi_{16}\}$ are the zeroes  of the section of the first cotangent bundle
$T^1_R$ of $R,$ determined by the first order embedded deformation associated to $\mathcal S\to T,$
\cite[pp. 644-647]{clm}. By blowing-up $\mathcal S$ at these points and by contracting the corresponding 
exceptional components (all isomorphic to $\mathbb F_0$) on $R_2,$ we get a smooth family of surfaces $\mathcal X\to T,$ such that $\mathcal X_t\simeq\mathcal S_t$ and 
$\mathcal X_0=R_1\cup \tilde{R_2},$ where $\tilde{R_2}$ is the blowing-up of $R_2$ at the points  $\{\xi_1, \ldots, \xi_{16}\},$
with new exceptional curves $E_1, \ldots,\,E_{16}.$ We will name $\{\xi_1, \ldots, \xi_{16}\}$ {\it the special points of} $E.$

\begin{lemma}\label{limiti}
Let $R=R_1\cup R_2\subset\mathbb P^{\p}$ as above. Then, for every $n\geq 1$ if $\p\geq 5$ and 
$n\geq 2$ if $\p=3,4,$ there exists 
a one parameter family of curves
 $C=C_{1}\cup C_{2}\in |\mathcal O_{R}(nH)|$ such that:
 \begin{enumerate}
\item $C_i\subset R_i$ and $C_i=  C_i^1\cup...\cup C_i^{n-1}\cup D_i\cup L_i,$
 where $C_i^j,$ $D_i$ and $L_i$ are smooth rational curves with:
  \begin{itemize}
   \item $C_i^j\sim\sigma_i,$ $D_i\sim F_i,$ and $L_i\sim \sigma_i+(nl-1)F_i$ if $\p=2l+1$ is odd,  
   \item $C_i^j\sim\sigma_i+F_i,$  $D_i\sim F_i$ and $L_i\sim \sigma_i+(nl-n)F_i$ if $\p=2l$ 
   is even,
\end{itemize}
for every $1\leq j\leq n-1$ and $i=1,2;$\\
\item there exist distinct points $p,q,q_1,...,q_{2n}\in E,$ where $p$ is a general point  
and $q,q_1,...,q_{2n}$ are determined by the following relations:
\begin{itemize}
\item if $\p=2l+1$ is odd, $i=1,2$ and $1\leq j\leq n-1,$ then
\begin{eqnarray}
\,\,\,\,\,\,\,\,\,\,\,\,\,\,\,\,\,\,D_1\cap E\,\, (\mbox{resp.}\,\, D_2\cap E)& =&\left\{\begin{array}{l}
p+q_{2n}, \,\mbox{if}\,\, n \,\,\mbox{is odd}, \,(\mbox{resp. if}\,\, n\,\, \mbox{is even}),\\
p+q_{2n-1},\,\mbox{if}\,\, n\,\, \mbox{is even}, \,(\mbox{resp. if}\,\, n\,\, \mbox{is odd}),\\
\end{array}\right.\label{odd1}\\
C_1^j\cap E\,\,(\mbox{resp.}\,\, C_2^j\cap E)&=&\left\{\begin{array}{l}
q_{2j-1}+q_{2j+1},\,\mbox{if}\,\, j \,\,\mbox{is even}\,(\mbox{resp. if}\,\, j\,\, \mbox{is odd}),\\
q_{2j}+q_{2j+2},\,\mbox{if}\,\, j \,\,\mbox{is odd}\,(\mbox{resp. if}\,\, j\,\, \mbox{is even})\\
\end{array}\right. \label{odd2}\\
\mbox{and}\,\,\,\,\,\, L_i\cap E &=& (2nl-3)q+2p+q_i;
\end{eqnarray}
\item if $\p=2l$ is even, $i=1,2$ and $1\leq j\leq n-1,$ then
\begin{eqnarray}
D_1\cap E\,\, (\mbox{resp.}\,\, D_2\cap E)& =& p+q_{2n}\,\,(\mbox{resp.}\,\,  p+q_{2n-1}),\label{even1}\\
C_1^j\cap E &=& q_{2j-1}+2q_{2j},\,\,1\leq j\leq n-1,\label{even2}\\
C_2^j\cap E &=& 2q_{2j}+q_{2j+1},\,\,1\leq j\leq n-2,\label{even3}\\
C^{n-1}_2\cap E &=& 2q_{2n-2}+q_{2n}\,\,\,\,and\label{even4}\\
 L_1\cap E\,\,(\mbox{resp.}\,\, L_2\cap E ) &=& (2nl-2n-2)q+2p+q_{2n-1}\label{even5}\\
&& ( \mbox{resp.}\,\,  (2nl-2n-2)q+2p+x,\,\,\mbox{where}\nonumber\\
&&x=q_2\,\,\mbox{if}\,\,
n=1\,\,\mbox{and}\,\,x=q_1\,\,\mbox{if}\,\,n>1);\nonumber
\end{eqnarray}
\end{itemize}
\item the singularities of $C$ on $R\setminus E$ are nodes, $C$ has a quadruple point
analytically equivalent to \eqref{eq: quadruple-point} at $p\in E$ and tacnodes and nodes on $E\setminus p.$
In particular, $C$ has a $(2nl-3)$-tacnode at $q$ and nodes at $q_1,...,q_{2n},$  if $\p=2l+1;$ 
$C$ has a $(2nl-2n-2)$-tacnode at $q,$ a simple tacnode at $q_{2k}$ 
and nodes at $q_{2n},q_{2n-1}$ and $q_{2k-1},$  for every $k=1,...,n-1,$  if $\p=2l$.
\end{enumerate}
\end{lemma}
\begin{proof} We first consider the case $\p=2l+1$ odd. Recall that, in this case, we have that $R_1\simeq R_2\simeq\mathbb F_0,$
$\mathcal O_{R_i}(H)=\mathcal O_{R_i}(\sigma_i+l F_i),$ where $|\sigma_i|$ and $|F_i|$
are the two rulings on $R_i,$ and the linear equivalence class of $E$ on $R_i$ is
$E\sim_{R_i}2\sigma_i+2 F_i,$ $i=1,2.$

Let $p$ and $q$ be two distinct points of $E= R_1\cap R_2\subset R$ and let us denote by
$W_{2,2nl-3}^i(p,q)\subset |\sigma_i+(nl-1)F_i|$ the family of divisors tangent to $E$ at $p$ and $q$ with
multiplicity $2$ and $2nl-3,$ respectively. Then $W_{2,2nl-3}^i(p,q)\subset |\sigma_i+(nl-1)F_i|$
 is a linear system of dimension $$\dim(W_{2,2nl-3}^i(p,q))\geq \dim(|\sigma_i+(nl-1)F_i|)-2nl+1=0.$$
Moreover, it is very easy to see that the general element of $W_{2,2nl-3}^i(p,q)$ corresponds to an irreducible and reduced 
smooth rational curve. By applying \cite[Proposition 2.1]{harris}, one shows that, in fact, $\dim(W_{2,2nl-3}^i(p,q))=0.$ We deduce that 
the variety $W_{2,2nl-3}^i\subset |\sigma_i+(nl-1)F_i|,$ parametrizing divisors tangent to $E$ at two distinct 
points with multiplicity $2$ and $2nl-3$ respectively, has dimension $2.$ Furthermore, looking at the tangent space 
to  $W_{2,2nl-3}^i$ at its general element, it is easy to prove that, if $p$ and $q$ are two general points of $E,$
then the unique divisor $L_i\in |\sigma_i+(nl-1)F_i|,$ tangent to $E$ at $p$ with multiplicity $2$ and to $q$ with multiplicity
$2nl-3,$ intersects $E$ transversally at a point $q_i\neq p,\,q.$ 

Now let $p$ be a general point of $E$ and let $D_i\subset R_i$ be the fibre $D_i\sim F_i$ passing through 
$p.$ Because of the generality of $p$ and the fact that $ED_i=2,$ we may assume that $D_1$ and $D_2$ 
intersect $E$ transversally at a further point. Again by the generality of $p,$ we may assume that the fibre
 $C^{n-1}_1 \sim \sigma_1$ (resp. $C^{n-1}_2 \sim \sigma_2$)
passing through the point of $D_2\cap E$ (resp. $D_1\cap E$), different from $p,$ 
intersects transversally $E$ at a further point.
Let $C^{n-2}_2$ (resp. $C^{n-2}_1$) be
the fibre of $|\sigma_2|$ (resp. $|\sigma_1|$) passing through this point. We repeat this argument $n-1$ times,
getting $2n$ points $q_1,...,q_{2n}$ of $E$ and fibres $C^{j}_i\sim\sigma_i,$ with $1\leq j\leq n-1,$
verifying relations \eqref{odd1} and \eqref{odd2}. From what we proved above,
by using that 
\begin{itemize}
\item the family of divisors in $|\sigma_1+(nl-1)F_1|$ tangent to $E$ at $p,$ with 
multiplicity $2,$ and at a further point, with multiplicity $2nl-3,$ has dimension $1;$ 
\item the point $p$ is general and the points $q_1,...,q_{2n}$ are determined by $p$
by the argument above;
\end{itemize}
we find that there exists a point $q$ such that the unique divisor $L_1\in W^1_{2,2nl-3}(p,q)$
passes through $q_1.$ Now $D_1+L_1+C_1^1+...+C_1^{n-1}\in |\mathcal O_{R_1}(nH)|$
and hence $(D_1+L_1+C_1^1+...+C_1^{n-1})E=(2nl-3)q+3p+q_1+\dots+q_{2n}\in |\mathcal O_{E}(nH)|.$
This implies that there exists on $R_2$ a divisor 
$$L_2\sim \sigma_2+(nl-1)F_2\sim (nH)|_{R_2}-D_2-C_2^1-...-C_2^{n-1}$$
cutting on $E$ the divisor $(2nl-3)q+2p+q_2.$ Moreover, $L_2$ is uniquely determined 
by the equality $\dim(|O_E(\sigma_2+(nl-1)F_2)|=\dim(|\mathcal O_{R_2}(\sigma_2+(nl-1)F_2)|)=2nl-1.$
Now, if $C_i=  C_i^1\cup...\cup C_i^{n-1}\cup D_i\cup L_i\subset R_i,$ then there exist only finitely many curves
like $C=C_1\cup R_2\subset R$ and passing through a fixed general point $p\in E.$  
 If $p$ varies on $E,$ then the curves $C,$ constructed in this way, move in a one parameter family
of curves $\mathcal W\subset |\mathcal O_R(n)|.$ By construction and by Lemma \ref{lemma: analytic-type},
 the curve $C$ has a space quadruple point of analytic equations  \eqref{eq: quadruple-point} at $p$ and nodes or tacnodes
 at points $q,q_1,...,q_{2n},$ as in the statement.
It remains to prove that, if $p$ is general or, equivalently,
if $[C]\in\mathcal W$ is general, then the singularities of $C$ on $R\setminus E$ are nodes.
Since $D_iL_i=D_iC_i^j=1,$ for every $i$ and $j,$ this is equivalent to showing that 
 $L_i$ intersects transversally $C_i^j,$ for every $i$ and $j.$ If $n=1$ there is nothing to show.
Just to fix ideas, we prove the statement for $n=2.$ Our argument trivially extends to the general case.
Let $[C]\in \mathcal W\subset |\mathcal O_R(2)|$ be a general element.
First observe that, since $\mathcal W$ is contained in the equisingular deformation 
locus $ES(C)\subset |\mathcal O_R(2H)|$ of $C,$  
it follows that, if $D$ is the curve corresponding to a point $[D]$ of the tangent
space $T_{[C]}\mathcal W,$ then $D$ has at $q$ a $(2nl-4)$-tacnode 
(cf. \cite[Proof of Theorem 3.3]{galati-knutsen}). Now assume that $L_1$ intersects
$C_1^1$ at $r\leq 2l-1$ points $\{x_i\}.$ Let $m_i\geq 1$ be the intersection multiplicity of $L_1$ with $C_1^1$ at $x_i.$ 
Then, the analytic equation of $C$ at $x_i$ is $y^2=x^{2m_i}$
and, by \cite[Proposition (5.6)]{diaz_harris}, the localization at $x_i$ of the equisingular deformation ideal of $C$ is $(y,x^{2m_i-1}).$
Thus $D$ must be tangent to $C^1_1$ at every point $x_i$ with multiplicity $2m_i-1.$
Similarly, $D$ must contain the intersection point of $C_1^1$ with $D_1.$ It follows that the cardinality
of intersection of $D|_{R_1}$ and $C_1^1$ is given by
$$
\sum_i(2m_i-1)+1=2l-1+\sum_im_i-r+1>2l,\,\,\mbox{if}\,\,m_i\geq 2\,\,\mbox{for some}\,\, i.
$$
We deduce that $C_1^1\subset D|_{R_1}.$ Using again that $D$ passes through every node
of $C$ and that, by Lemma \ref{lemma: analytic-type} and Remark \ref{eq-eq-in}, 
the curve $D|_{R_i}$ must be tangent to $E$ at $p,$ for $i=1,2,$ we obtain that $D$ contains the points 
$C_1^1\cap E, C_2^1\cap D_2$ and $p$ of $D_2.$ Thus, since the intersection number $D_2D|_{R_2}=2,$ 
we have that $D_2\subset D|_{R_2}.$  It follows that the analytic equations of $D\subset \mathcal X$ at $p$
are given by \eqref{equisingulardeformationquad2}, where $c=0$ and $u^\prime_z(x,y,z)=zu_z^{\prime\prime}(x,y,z).$ 
In particular, we find that $D|_{R_1}$ has a node at $p$ with one branch tangent to $E$
and the other one tangent to $D_1.$ Again, we find that $D|_{R_1}$ contains at least three points of $D_1,$
counted with multiplicity, and hence $D_1\subset D|_{R_1}.$ This implies, by repeating the same argument,
that $C_2^1\subset D|_{R_2}.$ Thus, for every $i=1,2,$ we have that $D|_{R_i}=D_i^\prime\cup D_i\cup C_i^1,$
where $D_i^\prime\sim \sigma_i+(2l-1)F_i,$ $D_i^\prime$ is tangent to $E$ at $p$ with multiplicity $2$ and at $q$ with multiplicity $2l-4$ and, finally, $D_i^\prime$ contains $q_i.$ 
But there is a unique divisor in $|\sigma_i+(2l-1)F_i|$ with these properties. We deduce that $D_i^\prime=L_i,$
for every $i=1,2,$ $D=C$
and $T_{[C]}\mathcal W=\{[C]\},$ getting a contradiction. This completes the proof
in the case $\p=2l+1$.

We now consider the case $\p=2l,$ where $R_1\simeq R_2\simeq \mathbb F_1,$
$\mathcal O_{R_i}(H)=\mathcal O_{R_i}(\sigma_i+l F_i),$ with $\sigma_i^2=-1$ and $F_i^2=0,$
and $E\sim_{R_i}2\sigma_i+3 F_i,$ for every $i=1,2.$
The proof of the lemma works  as in the previous case, except for the construction
of the curves $C_i^j\sim\sigma_i+F_i.$ Let
$p\in E$ be a general point and let $D_i\sim F_i$ be the fibre passing through $p,$ for $i=1,2.$ 
Because of the generality of $p,$
the curve $D_i$ intersects $E$ at a further point, say $q_{2n}$ if $i=1$ and $q_{2n-1}$ if $i=2.$
Now the curves in $|\sigma_2+F_2|$ passing through $q_{2n}$ cut out on $E,$
outside $q_{2n},$ a $g^1_2$ having, because of the generality of $p,$  four simple ramification
points. Let $C^{n-1}_2$ be one of the four curves in  $|\sigma_2+F_2|$ passing through $q_{2n}$
and simply tangent to $E$ at a further point $q_{2n-2}\neq q_{2n-1}, q_{2n}.$ Then, denote by 
$C_1^{n-2}\sim \sigma_1+F_1$ the unique curve tangent to $E$ at $q_{2n-2}$ and let $q_{2n-3}$ be
the further intersection point of $C^{n-2}_1$ with $E.$ Now repeat the same argument until you get 
curves $C^j_i,$ with $i=1,2$ and $j=1,...,n-1,$ and points $q_1,...,q_{2n}$ satisfying relations   
\eqref{even1}-\eqref{even4}. Again by the generality of $p,$ there exists a point $q$
such that the unique (smooth and irreducible) divisor $L_2\in |\sigma_2+(nl-n)F_2|,$ tangent
to $E$ at $p$ with multiplicity $2$ and at $q$ with multiplicity $2nl-2n-2,$ passes through $q_1.$
It follows that there exists a unique (smooth and irreducible) divisor $L_1\in |\sigma_1+(nl-n)F_1|$
passing through $q_{2n-1}$ and 
tangent to $E$ at $p$ and $q$ with multiplicity $2$ an $2nl-2n-2$ respectively. The curve we constructed has tacnodes, nodes and a space quadruple point of analytic
equations \eqref{eq: quadruple-point} on $E$, as desired. To see that the singularities of this curve outside $E$
are nodes, argue as in the case $\p=2l+1.$
 \end{proof}
 
 We may now prove Theorem \ref{main-theorem}.

\begin{proof}[Proof of Theorem \ref{main-theorem}]
Let $T^1_R$ be the first cotangent bundle of $R,$ defined by the standard
exact sequence 
\begin{equation}\label{normal_bundle_R}
\xymatrix{
0  \ar[r] &      \Theta_R  \ar[r] &      \Theta_{\mathbb P^{\p}}|_R \ar[r] 
& \mathcal N_{R|\mathbb P^{\p}} \ar[r] &   T^1_R \ar[r] &     0. }
\end{equation}
Since $R$ is a variety with normal crossings, by \cite[Section 2]{F}, we know that $T^1_R$
is locally free of rank one. In particular, 
$T^1_R\simeq \mathcal N_{E|R_1}\otimes\mathcal N_{E|R_2}$ is a degree $16$
line bundle on $E.$ Fix a general divisor $\xi_1+...+\xi_{16}\in |T^1_R|.$
Since the family of curves
constructed in the previous lemma is a one parameter family, we may always assume
that there exists a curve $C=C_{1}\cup C_{2}\in |\mathcal O_{R}(nH)|,$ with 
$C_i=  C_i^1\cup...\cup C_i^{n-1}\cup D_i\cup L_i,$ as in the statement of Lemma
\ref{limiti}, such that $q_{2n}=\xi_1$ and $q_j\neq \xi_k,$
for every $j\leq 2n-1$ and $k\leq 16.$ Now, by \cite[Corollary 1]{clm}, we have that
the induced map $H^0(R,\mathcal N_{R|\mathbb P^{\p}})\to H^0(R,T^1_R)$ is surjective.
By \cite[Theorems 1 and 2]{clm} and related references (precisely, \cite[Remark 2.6]{F}
and \cite[Section 2]{GH}) and because of the generality of $\xi_1+...+\xi_{16}\in |T^1_R|,$ 
it follows that there exists a deformation $\mathcal{S} \to T$ of $\mathcal S_0=R$ whose general fiber
is a smooth projective  $K3$ surface $\mathcal S_t$ in $\mathbb P^{\p}$ with 
$\Pic (\mathcal S_t) \cong \mathbb Z [\mathcal O _{\mathcal S_t}(1)]\cong \mathbb Z [H]$ and 
such that $\mathcal S$ is singular exactly at the
 points $\xi_1,...,\xi_{16}\in E.$
Let $\mathcal X\to T$ be the smooth family of surfaces  obtained by blowing-up $\xi_1,...,\xi_{16}$
and by contracting the corresponding exceptional components on $R_2,$ in such a way  that
$\mathcal X_0=R_1\cup \tilde{R_2},$ where $\tilde{R_2}$ is the blowing-up of $R_2$ at the points  
$\{\xi_1, \ldots, \xi_{16}\},$ with new exceptional curves $E_1, \ldots,\,E_{16}.$
Let  us denote by $\tilde C$  
the proper transform of $C$  and by
$\pi^*(C)=\tilde C\cup E_{1}$ the pull-back of $C$ with respect to
$\pi:\mathcal X\to\mathcal S.$ Now $\pi^*(C)$ has one more node at the point $x\in E_{1}\cap \tilde C$ on $\tilde R_2\setminus E.$
We want to prove the existence of irreducible curves $C_t\in|\mathcal O_{\mathcal X_t}(nH)|$
with the desired singularities by deforming of the curve $\pi^*(C).$ The irreducibility of $C_t$ easily follows from the 
fact that $\mathcal O_{\mathcal X_t}(H)$ is indivisible.

We first consider the case $p=2l+1.$ In this case  the singularities of the curve $\pi^*(C)$ are given by
\begin{itemize}
\item $2(n-1)(nl-1)$ nodes $y_1^i,...,y^i_{(n-1)nl},$ on $R\setminus E,$ arising from the intersection of the 
curves $C_{i}^{j}$, for $1 \leq j \leq n-1,$ with $L_i,$ for every $i=1,\,2;$
\item a node at $x\in E_1;$
\item $2n-2$ nodes $z_1^i,...,z_{n-1}^i,$ $i=1,2,$ arising from the intersection of $D_i$ with $C_i^j,$ for every $j\leq n-1$ and $i=1,2;$
\item $2n$  nodes at $q_1,...,q_{2n}\in E,$ where now $q_{2n}=E_1\cap E;$ 
\item a $(2nl-3)$-tacnode at $q$ and
\item a space quadruple point of analytic equations \eqref{eq: quadruple-point} at $p.$
\end{itemize}
Now, as we already observed in the proof of the previous lemma,  the tangent space
$T_{[\pi^*(C)]}ES(\pi^*(C))$ to the equisingular deformation locus  of $\pi^*(C)$ in $|\mathcal O_{\mathcal X_0}(nH)|$
is contained in the linear system of divisors $D=D^1\cup D^2\in|\mathcal O_{\mathcal X_0}(nH)|$ passing through 
the nodes of $\pi^*(C)$ on $\mathcal X_0\setminus E$ ($x$ included); having a $(2nl-4)$-tacnode at the point $q$ of $\pi^*(C)$
and having local analytic equations given by \eqref{equisingulardeformationquad2} at $p$.
This implies that, if $D\in T_{[\pi^*(C)]}ES(\pi^*(C)),$ then $D|_{R_2}$ contains the irreducible component of
$\tilde C$ passing through $x.$ It follows that $E_1\subset D|_{R_2}$ and so on, until, doing the same
local analysis of $D$ at $p$ as in the proof of the previous lemma, we find that
$$\dim(T_{[\pi^*(C)]}ES(\pi^*(C)))
=\{[\pi^*(C)]\}.$$
Using that the nodes of $\pi^*(C)$ at the points $q_i$ are trivially
preserved by every section of $H^0(C,\N_{\pi^*(C)|\mathcal X_0}),$ we deduce the injectivity  
of the standard morphism
$$\Phi:H^0(\pi^*(C),\N_{\pi^*(C)|\mathcal X})\to T,$$ where
$$T= \oplus_{j,i}T^1_{\pi^*(C),y_j^i}\oplus_{j,i}T^1_{\pi^*(C),z_j^i} \oplus T^1_{\pi^*(C),x}\oplus T^1_{\pi^*(C),q}\oplus T^1_{\pi^*(C),p}.$$
In particular, we have that $\Phi$ has image of dimension 
$$
\dim(Im(\Phi))=h^0(\pi^*(C),\N_{\pi^*(C)|\mathcal X})=2n^2l+2=\dim(|\mathcal O_{\mathcal X_t}(nH)|)+1.$$ 
Morever, by Corollary \ref{ingenerale} and by  \cite[Corollary 3.6]{galati-knutsen},
we know that  the image of 
the morphism $\Phi$ must be contained in 
$$T^\prime=\oplus_{j,i}T^1_{\pi^*(C),y_j^i}\oplus_{j,i}T^1_{\pi^*(C),z_j^i}\oplus T^1_{\pi^*(C),x}
\oplus H_p\oplus H_q\subset T,$$ 
where $H_p\subset T^1_{\pi^*(C),p}$ and $H_q\subset T^1_{\pi^*(C),q}$ are linear spaces of dimension $5$
and $2nl-3$ respectively. We first study the map $\Phi$ in the case that $2nl-3=1$ i.e. for $n=2$ and $l=1$ or $n=1$ and $l=2.$
In this case the curve
$\pi^*(C)$ has a $1$-tacnode, i.e. a node, at $q\in E$ and $H_q=T^1_{\pi^*(C),q}.$
Using again that every node of $\pi^*(C)$ on $E$
is trivially preserved by every section of $H^0(\pi^*(C),\N_{\pi^*(C)|\mathcal X_0}),$ we obtain that the induced
morphism
$H^0(\pi^*(C),\N_{\pi^*(C)|\mathcal X})\to \oplus_{j,i}T^1_{\pi^*(C),y_j^i}\oplus_{j,i}T^1_{\pi^*(C),z_j^i} \oplus T^1_{\pi^*(C),x}
\oplus H_p$
is injective. In fact this morphism is also
 surjective by virtue of the equality $$2(n-1)(nl-1)+2n-2+1+5=2n^2l-2nl-2n+2+2n+4=2n^2l+2.$$ 
By versality and by the proof of Theorem \ref{th: triple-point}, it follows that 
we may deform $\pi^*(C)$ to an irreducible curve $C_t\in |\mathcal O_{\mathcal X_t}(nH)|$
such that $T_{C_t}ES(C_t)=0,$ by preserving all
nodes of $\pi^*(C)$ on $\mathcal X_0 \setminus E$ and by deforming the singularity of $\pi^*(C)$ at $p$ to an ordinary
triple point. 
Now we study the morphism $\Phi$ under the assumption that $2nl-3\geq 2.$ In this case
$\Phi$ is not surjective. More precisely, 
we have that
$$
\dim(T^\prime)=2(n-1)(nl-1)+2n-2+1+5+2nl-3=2n^2l+3$$
and $Im(\Phi)=\Phi(H^0(\pi^*(C),\N_{\pi^*(C)|\mathcal X}))$ 
 is a hyperplane in $T^\prime,$ containing  the image $\Phi(H^0(\pi^*(C),\N_{\pi^*(C)|\mathcal X_0}))$
of $H^0(\pi^*(C),\N_{\pi^*(C)|\mathcal X_0})$ as a codimension $1$ subspace. 
Now, using on $T^1_{\pi^*(C),p}$ the affine coordinates $(b_1,b_2,b_3,a_1,a_2,a_3,a_4)$
introduced in the proof
of Lemma \ref{lemma: analytic-type}, by the proof of Theorem \ref{th: triple-point}, by \cite[Proof of Theorem 3.3]{galati-knutsen}
and by a straightforward dimension count, we have that $\Phi(H^0(\pi^*(C),\N_{\pi^*(C)|\mathcal X_0}))$
coincides with the $(2n^2l+1)$-plane
$$ 
\oplus_{j,i}T^1_{\pi^*(C),y_j^i}\oplus_{j,i}T^1_{\pi^*(C),z_j^i}\oplus T^1_{\pi^*(C),x}
\oplus \Gamma_p\oplus \Gamma_q,
$$
where $\Gamma_p\subset H_p\subset T^1_{\pi^*(C),p}$ is the $4$-plane of equations $b_1=b_2=b_3=0$
and $\Gamma_q\subset H_q\subset T^1_{\pi^*(C),q}$ is the $(2nl-4)$-plane parametrizing $(2nl-3)$-nodal curves. 
In particular, we have that 
$$
\Phi(H^0(\pi^*(C),\N_{\pi^*(C)|\mathcal X}))=\oplus_{j,i}T^1_{\pi^*(C),y_j^i}\oplus_{j,i}T^1_{\pi^*(C),z_j^i}\oplus T^1_{\pi^*(C),x}
\oplus\Omega,
$$ 
where $\Omega$ is a hyperplane in $H_p\oplus H_q$ such that $\Gamma_p\oplus\Gamma_q\subset \Omega.$
It trivially follows that the projection maps $\rho_p:\Omega\to H_p$ and $\rho_q:\Omega\to H_q$ are surjective.
Now, by the proof of Theorem \ref{th: triple-point}, we know that the locus of curves with a triple point
in $H_p$ is the smooth curve $T$ 
having equations $a_1=a_2=a_3=b_2=b_3=4b_1-a_4^2=0$ and 
intersecting $\Gamma_p$ only at $\underline 0$. 
Similarly, by \cite[proof of Theorem 3.3]{galati-knutsen}, for every $(m-1)$-tuple of non-negative
 integers $d_2,\ldots,d_m$ such that $\sum_{k=2}^m(k-1)d_k=2nl-4,$ the locus  
 $V_{\small 1^{d_2},\,2^{d_3},...,\,{m-1}^{d_m}}\subset H_q$ of points corresponding
 to curves with $d_k$ singularities of type $A_{k-1},$ for every $k,$ is a reduced (possibly
 reducible) curve intersecting $\Gamma_q$ only at $\underline 0.$ It follows that,
 for every $(m-1)$-tuple of non-negative
 integers $d_2,\ldots,d_m$ such that $\sum_{k=2}^m(k-1)d_k=2nl-4,$
  the locus $(T\times V_{\small 1^{d_2},\,2^{d_3},...,\,{m-1}^{d_m}})\,\cap \Omega$
   is a reduced (possibly reducible) curve whose parametric equations may be explicitly 
   computed by arguing exactly as in \cite[proof of Lemma 4.4, p.381-382]{ch}.   This proves, by versality,
that we may deform $\pi^*(C)$ to an irreducible curve $C_t\in |\mathcal O_{\mathcal X_t}(nH)|,$ by preserving all
nodes of $\pi^*(C)$ at $y_j^i,$ $z_j^i$ and $x,$ by deforming the singularity of $\pi^*(C)$ at $p$ to an ordinary
triple point and the $(2nl-3)$-tacnode of $\pi^*(C)$ to $d_k\geq 0$ singularities of type $A_{k-1}$ for every
$m$-tuple of integers $d_k$ such that $\sum_{k=1}^m((k-1)d_k)=2nl-4.$
In particular, if $k=2$ and $d_2=2nl-4,$ the corresponding curve $C_t$ is an elliptic curve in
$|\mathcal O_{\mathcal X_t}(nH)|$ with an ordinary triple point and nodes as singularities.
Finally, by the injectivity of the morphism $\Phi,$ we have that, if $ES(C_t)$ is the equisingular deformation
locus of $C_t$ in $|\mathcal O_{\mathcal X_t}(nH)|,$ then $\dim(T_{[C_t]}ES(C_t))=0$ and this implies that
 the singularities of $C_t$ may be smoothed independently, i.e.
 that the standard morphism $H^0(C_t,\N_{C_t|\mathcal X_t})\to T^1_{C_t}$ is surjective. This proves the theorem for $\p=2l+1.$

In the case $\p=2l$ the singularities of the curve $\pi^*(C)$ are given by
\begin{itemize}
\item $2(n-1)(nl-n)$ nodes $y_1^i,...,y^i_{(n-1)nl},$ on $R\setminus E,$ arising from the intersection of the 
curves $C_{i}^{j}$, for $1 \leq j \leq n-1,$ with $L_i,$ for every $i=1,\,2;$
\item a node at $x\in E_1;$
\item $2n-2$ nodes $z_1^i,...,z_{n-1}^i,$ $i=1,2,$ arising from the intersection of $D_i$ with $C_i^j,$ for every $j\leq n-1$ and $i=1,2;$
\item $n-1$  simple tacnodes at $q_2,q_4,...,q_{2n-2}\in E$ and nodes at $q_1,q_3,...,q_{2n-1},q_{2n},$ where
now $q_{2n}=E_1\cap E;$
\item a $(2nl-2n-2)$-tacnode at $q;$
\item a space quadruple point of analytic equations \eqref{eq: quadruple-point} at $p.$
\end{itemize}
With the same argument as above, one may prove that the curve $\pi^*(C)$ may be deformed
to a curve $C_t\in |\mathcal O_{\mathcal X_t}(nH)|,$ by preserving all
nodes of $\pi^*(C)$ at $y_j^i,$ $z_j^i$ and $x,$ by deforming the singularity of $\pi^*(C)$ at $p$ to an ordinary
triple point, every simple tacnode of $\pi^*(C)$ on $E$ to a node and the $(2nl-2n-2)$-tacnode of $\pi^*(C)$ to $d_k$ singularity of type $A_{k-1}$ for every
$m$-tuple of non negative integers $d_k$ such that $\sum_{k=1}^m((k-1)d_k)=2nl-2n-3.$ The curve $C_t$
obtained in this way has the desired singularities, is a reduced point for the equisingular
deformation locus and its singularities may be smoothed independently. Finally, if we choose $k=2$ and $d_2=
2nl-2n-3,$ then $C_t$ is an elliptic irreducible curve
with a triple point and nodes as singularities.
\end{proof}

\begin{corollary}\label{main-corollary}
Let $(S,H)$ be a general primitively polarized $K3$ surface of genus ${\p}=p_a(H)$ as above.
 Then, for every $1\leq g\leq p_a(nH)-3$ and for every $(\p,n)\neq (4,1),$ 
 there exist reduced and irreducible curves in $|nH|$
 of geometric genus $g$ with an ordinary triple point and nodes as singularities
 and corresponding to regular points of their equisingular deformation locus.
\end{corollary}
In accordance with \cite{C1}, we do not expect the existence of rational curves in $|\mathcal O_S(nH)|$
with a triple point.
When $n=1$ and $\p\geq 5$, Theorem \ref{main-theorem} implies that the family in $|H|$ of curves with a 
triple point and $\delta_k$ singularities of type $A_{k-1}$ is non-empty whenever it has expected dimension
at least equal to one. The precise statement is the following.
\begin{corollary}\label{primitive-corollary}
Let $(S,H)$ be a general primitively polarized $K3$ surface of genus ${\p}=p_a(H)\geq 5.$ Then,
for every $(m-1)$-tuple of non-negative integers $d_2,\ldots,d_m$ such that
 \begin{equation*}
\sum_{k=2}^m(k-1)d_k=\p-5=\dim(|H|)-5,
\end{equation*}
there exist reduced irreducible curves $C$ in the linear system $|H|$ on $S$ 
 such that:
 \begin{itemize}
 \item $C$ has an ordinary triple point, a node and $d_k$ singularities
 of type $A_{k-1},$ for every $k=2,\ldots,m,$ and no further singularities;
 \item  $C$ corresponds to a regular point of the equisingular deformation locus $ES(C).$ Equivalently, 
 $\dim(T_{[C]}ES(C))=0.$ 
 \end{itemize} 
 Finally, the singularities of $C$ may be smoothed independently.
\end{corollary}
\begin{remark}
The case $\p=4$ and $n=1$ is the only case where the existence of elliptic curves with 
a triple point is expected but it is not treated in this paper. In this case, with the notation above, it is 
easy to show the existence of a unique curve $C\in |\mathcal O_{\mathcal X_0}(H)|$ with a space
quadruple point analytically equivalent to \eqref{eq: quadruple-point}. Because of the unicity,
the argument we used in the proof of Theorem \ref{main-theorem} 
to compute the dimension of the tangent space to the equisingular deformation locus
of $C$ does not apply. Actually, we expect that $\dim(T_{[C]}ES(C)) >0.$ Nevertheless,
it is easy to prove that $C$ deforms to curves $C_t\in |\mathcal O_{X_t}(H)| $
with a triple point as singularity on the general $K3$ surface $\mathcal X_t.$
But describing the equisingular deformation locus $ES(C_t)$
from the scheme-theoretic point of view seems to us to be a very difficult problem.  The
case $\p=4$ and $n=1$ will be treated in detail in an upcoming article on related topics.
\end{remark}

{}
\end{document}